\documentclass[english]{cccconf}

\usepackage[comma,numbers,square,sort&compress]{natbib}
\usepackage{epstopdf}
\usepackage{amsmath}
\usepackage{amsthm}
\usepackage{amssymb}
\usepackage{amsfonts}
\usepackage{bm}
\usepackage{bbm}
\usepackage{mathrsfs}
\usepackage{caption}
\usepackage{subcaption}

\usepackage{bm}

\usepackage[dvipsnames]{xcolor}

\usepackage[noadjust]{cite} 

\usepackage[hyphens]{url}

\newtheorem{theorem}{Theorem}
\newtheorem{lemma}{Lemma}
\newtheorem{remark}{Remark}

\newtheorem{assumption}{Assumption}

\newcommand{\mR}{\mathbb{R}}
\newcommand{\mS}{\mathbb{S}}

\newcommand{\trueQ}{\bar{Q}}
\newcommand{\trueq}{\bar{q}}

\DeclareMathOperator*{\mE}{\mathbb{E}}
\DeclareMathOperator*{\mP}{\mathbb{P}}

\newcommand{\st}{\text{s.t. }}

\DeclareMathOperator*{\tr}{tr}
\DeclareMathOperator*{\cov}{cov}

\newcommand{\mfx}{\bm{x}}
\newcommand{\mfy}{\bm{y}}
\newcommand{\mfz}{\bm{z}}
\newcommand{\mfw}{\bm{w}}
\newcommand{\mfu}{\bm{u}}

\def\utilde#1{\mathord{\vtop{\ialign{##\crcr
$\hfil\displaystyle{#1}\hfil$\crcr\noalign{\kern1.5pt\nointerlineskip}
$\hfil\tilde{}\hfil$\crcr\noalign{\kern1.5pt}}}}}

\begin{document}

\title{Statistically Consistent Inverse Optimal Control for Linear-Quadratic Tracking with Random Time Horizon}

\author{Han Zhang\aref{SJTU},
        Axel Ringh\aref{CHALMERS_AND_GU},
        Weihan Jiang\aref{SJTU},
        Shaoyuan Li\aref{SJTU},
        Xiaoming Hu\aref{KTH}
        }


\affiliation[SJTU]{Department of Automation, School of Electronic Information and Electrical Engineering, Shanghai Jiao Tong University, Shanghai, China
        \email{\{zhanghan\_tc, JiangWH9, syli\}@sjtu.edu.cn}}
\affiliation[CHALMERS_AND_GU]{Department of Mathematical Sciences, Chalmers University of Technology and University of Gothenburg, 412 96, Gothenburg, Sweden
        \email{axelri@chalmers.se}}
\affiliation[KTH]{Department of Mathematics, KTH Royal Institute of Technology, SE-100 44, Stockholm, Sweden
        \email{hu@kth.se}}

\maketitle

\begin{abstract}
The goal of Inverse Optimal Control (IOC) is to identify the underlying objective function based on observed optimal trajectories. It provides a powerful framework to model expert's behavior, and a data-driven way to design an objective function so that the induced optimal control is adapted to a contextual environment.
In this paper, we design an IOC algorithm for linear-quadratic tracking problems with random time horizon,
and prove the statistical consistency of the algorithm. More specifically, the proposed estimator is the solution to a convex optimization problem, which means that
the estimator does not suffer from local minima. This enables the proven statistical consistency to actually be achieved in practice.
The algorithm is also verified on simulated data as well as data from a real world experiment, both in the setting of identifying the objective function of human tracking locomotion.
The statistical consistency is illustrated on the synthetic data set, and the experimental results on the real data shows that we can get a good prediction on human tracking locomotion based on estimating the objective function. It shows that the theory and the model have a good performance in real practice. Moreover, the identified model can be used as a control target in personalized rehabilitation robot controller design, since the identified objective function describes personal habit and preferences.
\end{abstract}


\footnotetext{This work was partially supported by National Natural Science
Foundation (NNSF) of China under Grant 62103276, and partially by the Wallenberg AI, Autonomous Systems and Software Program (WASP) funded by the Knut and Alice Wallenberg Foundation.}

\section{Introduction}

Optimal control is a powerful framework in which control actions are selected in order to minimize some given objective function; see, e.g., one of the monographs \citep{anderson2007optimal, bertsekas2000dynamic}. In fact, many processes in nature have been observed to be optimal with respect to some criteria \citep{alexander1996optima}. However, in engineering applications of optimal control, a \emph{fundamental} problem is to design an appropriate objective function:  it needs to be adapted to the contextual environment in which the system is operating in order to induce an appropriate control response. This is a difficult task, which relies heavily on the designers imagination. One way to overcome this would be to, instead of designing the cost criteria, identify it from observations of an expert system that behaves ``optimally" in the environment. The latter is known as Inverse Optimal Control (IOC) \citep{kalman1964linear}, and has received considerable attention. In particular, the linear-quadratic problem has been studied in many different settings, including the infinite-horizon case in both continuous time \citep{boyd1994linear, anderson2007optimal} and discrete time \citep{priess2014solutions}, respectively, as well as the finite-horizon case in both continuous time \citep{li2018convex, li2020continuous} and discrete time \citep{keshavarz2011imputing, zhang2019inverse, yu2021system, zhang2021inverse}, respectively. More general underlying dynamics and objective functions have also been considered \citep{keshavarz2011imputing, hatz2012estimating, pauwels2016linear, rouot2017inverse, molloy2018finite, molloy2020online} and applied in areas such as, e.g.,  machine learning \citep{finn2016guided, kopf2017inverse}, and to model and analyze human locomotion of different forms \citep{mombaur2010human, westermann2020inverse}.

However, all of the above IOC frameworks have limitations, in particular if one wants to apply IOC in a real-world situation. More precisely, any real-world data would inevitably contain noise: it can be process noise, observation noise, or both. Therefore, from a robustness and accuracy perspective, it is important to have an unbiased estimator which is statistically consistent, i.e., that converges to the true underlying parameter values when the number of observation increases. To the best of our knowledge, none of the IOC frameworks with more general underlying dynamics and objective functions have considered this aspect. Furthermore, they all suffer from the fact that the estimation problems are nonconvex, and hence a globally optimal solution cannot be guaranteed in practice.
Regarding the literature on linear-quadratic IOC, most of them consider the stabilization problem. However, in many experimental set-ups, it is of greater interest to have the subject under investigation track a time-varying reference signal and to identify the corresponding objective function. Finally, real-world data can be of different time lengths, and this needs to be handled in a systematic way in order not to deteriorate the estimates.

In this work, we address these issues. More specifically, we consider the case of linear-quadratic, discrete-time IOC. The observations are obtained from expert systems tracking a given reference signal, which naturally also puts us in a finite-horizon setting since the reference signal is of finite length. Moreover, we model the data as arising from a system with process noise. Finally, in order to handle data of different time lengths, we also model the planing horizon in the forward problem as stochastic. Despite this, we can formulate the estimator as the solution to a convex optimization problem, and we show that it is statistically consistent. Lastly, we demonstrate the novelty and effectiveness of the method on a real-world experiment, where we identify the objective function of a person performing rotational motions of the elbow to track a given reference signal.

\textit{Notations:} $\mathbb{Z}_+$ denotes postive integers, and $\mathbb{S}^n_+$ denote the sets of $n\times n$ positive semi-definite matrices. We use $\succeq$ to denote the Loewner partial order on $\mathbb{S}^n_+$, i.e., for $G_1,G_2\in\mathbb{S}^n$, $G_1\succeq G_2$ means that $G_1-G_2\in\mathbb{S}^n_+$, and $G_1\succ G_2$ means that $G_1-G_2$ is strictly positive definite. $\|\cdot\|$ denotes the $l_2$-norm and $\|\cdot\|_F$ denotes the Frobenius norm.
$\bm{1}_m$ denotes an all-one vector of dimension $m$, and $\otimes$ denotes Kronecker product. Further, we use \textit{\textbf{italic bold font}} to denote stochastic vectors, and we use $\mE_{\mfx}[\cdot]$ and $\mE_{\mfx|\mfy}[\cdot]$ to denote the expected value under the distribution $\mP(\mfx)$ and the conditional distribution $\mP(\mfx | \mfy)$, respectively. The conditional covariance is denoted as $\cov_{\mfx, \mfy| \mfz}(\mfx, \mfy):=\mE_{\mfx, \mfy| \mfz}[(\mfx -\mE_{\mfx | \mfz}[\mfx])(\mfy-\mE_{\mfy | \mfz}[\mfy])^T]$.
Finally, $\overset{p}\rightarrow$ denotes convergence in probability.

\section{Problem formulation}
We start by introducing the mathematical formulation of the forward optimal control problem.
To this end, let $(\Omega,\mathcal{F},\mathbb{P})$ be a probability space that carries random vectors $\bar{\mfx}\in\mR^n$, $\{\mfw_t\in\mR^n\}_{t=1}^\infty$, and a random variable $\mathscr{N}\in\{\nu_1,\nu_1+1,\cdots,\nu_2\}\subset\mathbb{Z}_+$.
For each realization $(\bar{x},N)$ of the random element $(\bar{\mfx},\mathscr{N})$, corresponding to the initial position and planning horizon length, suppose that the tracking control of an agent is determined by a stochastic linear-quadratic control problem with respect to some known reference signal $\{ x_{t}^{r} \}_{t=1}^{\nu_2}$, namely,
\begin{subequations}\label{eq:stochastic_forward_problem}
\begin{align}
\min_{\substack{\mfx_{1:\nu_2} , \\ \mfu_{1:\nu_2}}}
  \;
  & \; \mE_{\mfw_{\nu_2-N+1:\nu_2-1}} \Big[\frac{1}{2} (\mfx_{\nu_2} -  x_{\nu_2}^{r})^{T} \trueQ  (\mfx_{\nu_2} -  x_{\nu_2}^{r}) \label{eq:stochastic_forward_problem_cost} \\
  & \;\; +\sum_{t = \nu_2 - N+1}^{\nu_2-1} [\frac{1}{2} (\mfx_{t} -  x_{t}^{r})^{T} \trueQ  (\mfx_{t} -  x_{t}^{r}) + \frac{1}{2}\|\mfu_{t}\|^2 \Big] \nonumber \\
  \st
  & \; \mfx_{t+1} = A\mfx_{t} + B(\mfu_{t}+\mfw_t), \; t = \nu_2\!-\!N\!+\!1:\nu_2\!-\!1, \label{eq:stochastic_forward_problem_dynamics} \\
  & \; \mfx_{\nu_2-N + 1} = \bar{x}, \label{eq:stochastic_forward_problem_init_cond}\\
  & \; \mfx_1=\ldots=\mfx_{\nu_2-N} = 0, \label{eq:stochastic_forward_problem_before_init_cond} \\
  & \; \mfu_1=\ldots=\mfu_{\nu_2-N} = 0. \label{eq:stochastic_control_stand_still}
\end{align}
\end{subequations}
Here, $\trueQ \in\mS_+^n$, and $\mfu_t\in\mR^m$ is the control signal. 
Notably, \eqref{eq:stochastic_forward_problem} means that the agent starts tracking the reference signal $\{x_t^r\}_{t=1}^{\nu_2}$ at the time instant $\nu_2-\mathscr{N}+1$, which is stochastic, and hence that the tracking is done over different planning horizons but with \textit{the same terminal cost} and \textit{running cost}.
More precisely, given the time-horizon length $\mathscr{N}=N$ and initial value $\bar{\mfx} = \bar{x}$, the distribution of the agent's optimal trajectory $\mP(\mfx_{1:N}^{\star}\mid \mathscr{N}=N,\bar{\mfx}=\bar{x})$ and optimal control $\mP(\mfu_{1:N-1}^{\star} \mid \mathscr{N}=N,\bar{\mfx}=\bar{x})$ are implicitly given by solving \eqref{eq:stochastic_forward_problem}, where we for convenience set $ \mfx_t = 0$ and $\mfu_t = 0 $ for time points $t$ before the tracking starts (see \eqref{eq:stochastic_forward_problem_before_init_cond} and \eqref{eq:stochastic_control_stand_still}).
In this context, we make the following mild assumptions on the stochastic planning horizon.

\begin{assumption}[Stochastic planning horizon]\label{ass:planning_horizon}
The constant $\nu_2 \in \mathbb{Z}_+$ is known, and $\nu_2 \ge n+1$. Moreover, the probability distribution for $\mathscr{N}$ satisfies $\mP(\mathscr{N} \in [\nu_1, \nu_2])=1$, and $\mP(\mathscr{N} = \nu_2) > 0$, where $\nu_1 \in \mathbb{Z}_+$ and $\nu_1 \leq \nu_2$.
\end{assumption}

The above assumptions means that the longest possible planning horizon is known, that it is sufficiently long, and that it can actually be realized, i.e., that it has a nonzero probability.
Moreover, under mild regularity conditions on the probability distribution of $(\bar{\mfx}, \mathscr{N})$, the formulation in \eqref{eq:stochastic_forward_problem} defines joint probability distributions for $(\mfx_{1:\mathscr{N}}^\star,\mathscr{N},\bar{\mfx})$ and $(\mfu_{1:\mathscr{N}-1}^{\star},\mathscr{N},\bar{\mfx})$ (cf.~\citep[Thm.~5.3]{kallenberg1997foundations}). From now on, in order to ease the notation, we omit the ``star" in the superscript in the agents optimal states and control inputs.

Next, it is well-known that a linear-quadratic tracking problem can be rewritten as a linear-quadratic problem with a time-varying linear cost term. To this end,
by introducing $\trueq_t := -\trueQ x_{t}^{r}$, for $t  = 1, \ldots, \nu_2$, and omitting the constant terms, the objective function \eqref{eq:stochastic_forward_problem_cost} can be re-written as
\begin{align}
&\mE_{\mfw_{\nu_2-N+1:\nu_2-1}} \Big[\frac{1}{2} \mfx_{\nu_2}^T\bar{Q}\mfx_{\nu_2}+\bar{q}_{\nu_2}^T\mfx_{\nu_2}\nonumber\\
&+\sum_{t = \nu_2 - N+1}^{\nu_2-1} \frac{1}{2} \mfx_{t}\bar{Q}\mfx_t+\bar{q}_t^T\mfx_{t}+  \frac{1}{2}\|\mfu_{t}\|^2\Big].
\label{eq:stochastic_forward_problem_obj_rewrite}
\end{align}

The IOC problem we consider in this work is: \emph{given observations of trajectories generated as solutions to the forward problem \eqref{eq:stochastic_forward_problem}, form a statistically consistent estimate of the parameter $\trueQ$ in the objective function.} In order for this to be a well-posed problem, we need a number of assumptions.

\begin{assumption}[System dynamics]\label{ass:AB}
The system $(A,B)$ is controllable, $B$ has full column rank, and $A$ is invertible.
\end{assumption}
Controllability is a standard assumption, and the fact that $B$ has full column rank means that there are no redundant control signals (in particular, the system is not over-actuated).
The last assumption is motivated by the fact that a discrete-time system is often sampled from a continuous-time system, and the discretized $A$ matrix is always invertible. In addition, we also make the following assumptions regarding the random elements and the parameter that we want to estimate.

\begin{assumption}[I.I.D random variables]\label{ass:IID}
The random elements $(\bar{\mfx},\mathscr{N})$ and $\{\mfw_t\}_{t=1}^\infty$ are all independent. In addition, the random vectors $\{\mfw_t\}_{t=1}^\infty$ are identically distributed (I.I.D), $\mE[\mfw_t]=0$, $\forall t$, and $\cov(\mfw_t,\mfw_t)=\Sigma_w \succ 0$, where $\Sigma_w$ is a priori known and $\|\Sigma_w\|_F<\infty$.
\end{assumption}

\begin{assumption}[Persistent excitation]\label{ass:persistent_excitation}
It holds that $\cov_{\bar{\mfx}|\mathscr{N}=\nu_2}(\bar{\mfx},\bar{\mfx})\succ 0$ and $\mE[\|\bar{\mfx}\|^2]<\infty$.
\end{assumption}

\begin{assumption}[Bounded parameter]\label{ass:bounded_parameter}
The unknown ``true" parameter $\bar{Q}$ that governs the agents tracking behavior lives in the compact set
\[
\mathbb{G}(\varphi) :=\left\{\bar{Q}\in\mS^n_+ \mid \| \bar{Q} \|_F\le \varphi\right\},
\]
for some (potentially unknown) $0 < \varphi < \infty$.
\end{assumption}

\section{Main results}
The optimal control signal that solves \eqref{eq:stochastic_forward_problem} turns out to have the same form as in the deterministic case, i.e., without the process noise $\mfw_t$ in the dynamics \eqref{eq:stochastic_forward_problem_dynamics}. This can be seen by following a derivation similar to the one in \citep[Sec.~4.1]{bertsekas2000dynamic}. In particular, conditioned on $\mathscr{N} = \nu_2$ we have that for $t=1:\nu_2-1$, the optimal control signal is given by
\begin{equation}\label{eq:form_of-stochastic_u}
\mfu_t = -(B^T \bar{P}_{t+1}B + I)^{-1}\left(B^T\bar{P}_{t+1}A \mfx_t + B^T \bar{\eta}_{t+1}\right),
\end{equation}
where the sequence of $\bar{P}_{1:N}$ and $\bar{\eta}_{1:N}$ are generated by the Riccati iterations
\begin{subequations}\label{eq:generalized_riccati_iterations}
\begin{align}
 \bar{P}_{N} & = \trueQ, \label{eq:generalized_riccati_1}\\
\bar{P}_t & =  A^T \bar{P}_{t+1} A  + \trueQ \nonumber \\
 & \; - A^T \bar{P}_{t+1} B (B^T \bar{P}_{t+1} B  + I)^{-1}B^T \bar{P}_{t+1} A, \nonumber\\
 &\qquad t=1:N-1; \label{eq:generalized_riccati_2}\\
\bar{ \eta}_N &= \trueq_N,\label{eq:generalized_riccati_3}\\
 \bar{\eta}_t &= \left(A-B(B^T\bar{P}_{t+1}B+I)^{-1} B^T\bar{P}_{t+1}A\right)^T\bar{\eta}_{t+1}\nonumber\\
 &\:+ \trueq_t,\quad t=1:N-1.\label{eq:generalized_riccati_4}
\end{align}
\end{subequations}

In view of the special form of the linear-quadratic optimal control, we propose to solve the following convex optimization problem to reconstruct the ``true" parameter $\bar{Q}$:
\begin{subequations}\label{eq:stochastic_IOC_opt_pro}
\begin{align}
\min_{\substack{Q \in \mathbb{G}(\varphi)\\ \{ q_t \in \mR^n \}_{t=1:\nu_2} \\ \{P_{t} \in \mathbb{S}^n_+ \}_{t = 1:\nu_2},\\ \{ \eta_{t} \in\mathbb{R}^{n} \}_{t = 1:\nu_2},\\ \{ \xi_t \in \mR \}_{t=1:\nu_2-1}}} & \; \Psi(Q, q_{1:\nu_2},  P_{1:\nu_2}, \eta_{1:\nu_2},  \xi_{1:\nu_2-1})\nonumber\\
\st & \; P_{\nu_2} = Q, \quad \eta_{\nu_2}=q_{\nu_2}, \label{eq:stochastic_IOC_opt_pro_first_const}\\
& \; q_t = -Qx_{t}^{r}, \quad t=1:\nu_2\\
& 
H_t:=\begin{bmatrix}
\mathfrak{R}_t &\mathfrak{S}_t &g_t\\
\mathfrak{S}_t^{T}&A^TP_{t+1}A+Q-P_t &\beta_t\\
g_t^{T} &\beta_t^{T} &\xi_t
\end{bmatrix} \succeq 0,\nonumber\\
&\qquad \; t=1:\nu_2-1. \label{eq:stochastic_IOC_opt_pro_last_const}
\end{align}
where $\mathfrak{R}_t:=B^TP_{t+1}B+I$, $\mathfrak{S}_t := B^TP_{t+1}A$, $g_t := B^T\eta_{t+1}$, $\beta_t := q_t + A^T\eta_{t+1}-\eta_t$, and
\begin{align}
& \Psi(Q, q_{1:\nu_2}, P_{1:\nu_2}, \eta_{1:\nu_2}, \xi_{1:\nu_2-1}) :=\mE_{\mfx_{1:\nu_2},\mathscr{N}}\Big[\frac{1}{2}\mfx_{\nu_2}^TP_{\nu_2}\mfx_{\nu_2} \nonumber\\
&\; +\eta_{\nu_2}^T\mfx_{\nu_2} - \frac{1}{2}\mfx_{\nu_2-\mathscr{N}+1}^TP_{\nu_2-\mathscr{N}+1}\mfx_{\nu_2-\mathscr{N}+1} \nonumber\\
&\; -\eta_{\nu_2-\mathscr{N}+1}^T\mfx_{\nu_2-\mathscr{N}+1} + \sum_{t=\nu_2-\mathscr{N}+1}^{\nu_2-1}\Big(\frac{1}{2}\xi_t + \frac{1}{2} \mfx_t^TQ\mfx_t  \nonumber \\
&\; + q_t^T\mfx_t - \frac{1}{2}\tr(B^TP_{t+1}B\Sigma_w)\Big)\Big].
\label{eq:stochastic_IOC_objective_function}
\end{align}
\end{subequations}

Before we proceed, we have the following intermediate result that will be useful in the analysis.
\begin{lemma}\label{lem:stochastic_persistent_excitation}
Let $\tilde{\mfx}_t:=[\mfx_t^T,1]^T$ and $\bar{\tilde{\mfx}}_{t+1}:=[(A\mfx_t+B\mfu_t)^T,1]^T$.
Under Assumptions~\ref{ass:planning_horizon} and \ref{ass:persistent_excitation}, it holds that $\mE_{\mfx_t|\mathscr{N}=\nu_2}[\tilde{\mfx}_t\tilde{\mfx}_t^T]\succ 0$, $\mE_{\mfx_t|\mathscr{N}=\nu_2}[\bar{\tilde{\mfx}}_{t+1}\bar{\tilde{\mfx}}_{t+1}^T]\succ 0$, and $\mE[\|\tilde{\mfx}_t\|^2]<\infty$ for all $ t=1:\nu_2$.
\end{lemma}
\begin{proof}
Due to space limitation, the proof is omitted.
\end{proof}

\begin{theorem}\label{thm:stochastic_ioc_unique_opt_sol}
Let $(\bar{Q},\bar{q}_{1:\nu_2})$ be the ``true" parameters of the stochastic linear-quadratic regulator that governs the agent, and let $\bar{P}_{1:\nu_2}$, $\bar{\eta}_{1:\nu_2}$ be generated by the corresponding Riccati iterations \eqref{eq:generalized_riccati_iterations}. Under Assumptions~\ref{ass:planning_horizon}, \ref{ass:AB}, \ref{ass:IID}, \ref{ass:persistent_excitation} and \ref{ass:bounded_parameter}, $(\bar{Q},\bar{q}_{1:\nu_2},\bar{P}_{1:\nu_2},\bar{\eta}_{1:\nu_2})$ is the unique optimal solution to \eqref{eq:stochastic_IOC_opt_pro}.
\end{theorem}
\begin{proof}
To ease the notation, in the proof we omit the arguments of the objective function and simply write $\Psi(\cdot)$. Moreover, to prove the statement of the theorem we show  i) that $\Psi(\cdot)$ is bounded from below on the feasible region, ii) that the true parameter values attain this bound, and iii) that the solution is unique.

To this end, first note that \eqref{eq:stochastic_IOC_objective_function} can be written as
\begin{align}
&\Psi(\cdot)=\sum_{N = \nu_1}^{\nu_2}\mP(\mathscr{N}=N)\mE_{\mfx_{\nu_2-N+1:\nu_2}|\mathscr{N}=N}\left[\psi_N(\cdot) \right],\label{eq:stochastic_ioc_obj_rewriting}
\end{align}
where $\psi_N(\cdot) = \sum_{t=\nu_2-N+1}^{\nu_2-1}\psi_{t,N}(\cdot)$ and
\begin{align*}
\psi_{t,N}(\cdot) = & \; \frac{1}{2}\mfx_{t+1}^TP_{t+1}\mfx_{t+1} \! +\! \eta_{t+1}^T\mfx_{t+1} \! - \! \frac{1}{2}\mfx_{t}^TP_{t}\mfx_{t} \! - \! \eta_{t}^T\mfx_{t}\\
&\; +\! \frac{1}{2}\xi_t \!+\! \frac{1}{2} \mfx_t^TQ\mfx_t \!+\! q_t^T\mfx_t \!-\! \frac{1}{2}\tr(B^TP_{t+1}B\Sigma_w),
\end{align*}
and where the expectation in \eqref{eq:stochastic_ioc_obj_rewriting} is only taken over time points $\nu_2-N+1:\nu_2$ since the probability measure of $\mP(\mfx_{1:\nu_2-N}|\mathscr{N}=N)$ is marginalized out.
In view of the expression of $\psi_{t,N}(\cdot)$ and \eqref{eq:stochastic_forward_problem_dynamics}, we can further write \eqref{eq:stochastic_ioc_obj_rewriting} as
\begin{align}
&\Psi(\cdot)= \sum_{N = \nu_1}^{\nu_2}\mP(\mathscr{N}=N)\sum_{t=\nu_2-N+1}^{\nu_2-1}\mE_{\mfx_{t:t+1}|\mathscr{N}=N}\left[\psi_{t,N}(\cdot) \right] \nonumber\\
&=\sum_{N = \nu_1}^{\nu_2}\mP(\mathscr{N}=N) \sum_{t=\nu_2-N+1}^{\nu_2-1}\mE_{\mfw_t,\mfx_t|\mathscr{N}=N}\left[\psi_{t,N}(\cdot) \right] \label{eq:psi_probability_times_expectation}
\end{align}
where
\begin{align*}
&\mE_{\mfw_t,\mfx_t| \mathscr{N}=N} \! \left[\psi_{t,N}(\cdot) \right] = \! \mE_{\mfx_t|\mathscr{N}=N} \! \Big[\! \mE_{\mfw_t} \! \Big[\frac{1}{2}\left(A\mfx_t+B(\mfu_t+\mfw_t)\right)^T\\
&\times P_{t+1}\left(A\mfx_t+B(\mfu_t+\mfw_t)\right)+\eta_{t+1}^T\left(A\mfx_t+B(\mfu_t+\mfw_t)\right)\\
&-\frac{1}{2}\mfx_{t}^TP_{t}\mfx_{t} -\eta_{t}^T\mfx_{t}+\frac{1}{2}\xi_t + \frac{1}{2} \mfx_t^TQ\mfx_t + q_t^T\mfx_t\\
& -\frac{1}{2}\tr(B^TP_{t+1}B\Sigma_w)\Big]\Big],
\end{align*}
since by Assumption~\ref{ass:IID} the noise $\{\mfw_t\}_{t = 1}^{\infty}$ is independent of any other stochastic element.  Now, using Assumption~\ref{ass:IID} and the cyclic permutation property of the matrix trace operator, we know that
\begin{align*}
&\mE_{\mfw_t}[\mfw_t^TB^TP_{t+1}B\mfw_t]= \mE_{\mfw_t}[\tr(\mfw_t^TB^TP_{t+1}B\mfw_t)]\\
&=  \mE_{\mfw_t}[\tr(B^TP_{t+1}B\mfw_t\mfw_t^T)] =\tr(B^TP_{t+1}B\Sigma_w)
\end{align*}
which together with $\mE_{\mfw_t}[\mfw_t] = 0$ implies that
\begin{align*}
& \mE_{\mfw_t,\mfx_t| \mathscr{N}=N}\left[\psi_{t,N}(\cdot) \right] = \mE_{\mfx_t|\mathscr{N}=N}\Big[\frac{1}{2}\left(A\mfx_t+B\mfu_t\right)^T \\
& \times P_{t+1}\left(A\mfx_t+B\mfu_t\right)  +\eta_{t+1}^T\left(A\mfx_t+B\mfu_t\right) \\
& -\frac{1}{2}\mfx_{t}^TP_{t}\mfx_{t} -\eta_{t}^T\mfx_{t}+\frac{1}{2}\xi_t + \frac{1}{2} \mfx_t^TQ\mfx_t + q_t^T\mfx_t \Big]\\
&=\mE_{\mfx_t|\mathscr{N}=N}\Big[\frac{1}{2}
\begin{bmatrix}
\mfu_t^T &\mfx_t^T &1
\end{bmatrix}
H_t
\begin{bmatrix}
\mfu_t\\\mfx_t \\1
\end{bmatrix}-\frac{1}{2}\|\mfu_t\|^2\Big],
\end{align*}
where $H_t$ is the matrix given in \eqref{eq:stochastic_IOC_opt_pro_last_const}.
Next, note that for any feasible $P_t$ it holds that $P_t\in \mS^n_+$,  and hence $\mathfrak{R}_t = B^TP_{t+1}B+I\succ 0$ which means that it is invertible. In particular,
this means that we can take the Schur complement of $H_t$ with respect to $\mathfrak{R}_t$, which gives
\begin{align}
&H_t\backslash \mathfrak{R}_t= \label{eq:Ht_Schur_comp} \\
&\begin{bmatrix}
A^TP_{t+1}A+Q-P_t &\beta_t\\
\beta_t^T &\xi_t
\end{bmatrix}  -\begin{bmatrix}
\mathfrak{S}_t^T\\g_t^T
\end{bmatrix}
\mathfrak{R}_t^{-1}
\begin{bmatrix}
\mathfrak{S}_t &g_t
\end{bmatrix}. \nonumber
\end{align}
The above expression can therefore be further rewritten as 
\begin{align}
&\mE_{\mfw_t,\mfx_t| \mathscr{N}=N}\left[\psi_{t,N}(\cdot) \right] = \mE_{\mfx_t|\mathscr{N}=N}\Big[\frac{1}{2}
\begin{bmatrix}
\mfu_t^T &\mfx_t^T &1
\end{bmatrix}G_t\times\nonumber\\
&\quad\begin{bmatrix}
\mathfrak{R}_t\\&H_t\backslash \mathfrak{R}_t
\end{bmatrix}G_t^T
\begin{bmatrix}
\mfu_t\\\mfx_t \\1
\end{bmatrix}-\frac{1}{2}\|\mfu_t\|^2\Big],\label{eq:objective_function_schur_complement}
\end{align}
where
\[
G_t = \begin{bmatrix}
I\\
\mathfrak{S}_t^T\mathfrak{R}_t^{-1} &I\\
g_t^T\mathfrak{R}_t^{-1} & & 1
\end{bmatrix}.
\]
Now, using the notation $\bar{\tilde{\mfx}}_{t+1}$ (introduced in Lemma \ref{lem:stochastic_persistent_excitation}), we can expand \eqref{eq:objective_function_schur_complement} and therefore rewrite \eqref{eq:psi_probability_times_expectation} as
\begin{align}
&\Psi(\cdot) =\sum_{N = \nu_1}^{\nu_2}\Big\{\mP(\mathscr{N}=N)\sum_{t=\nu_2-N+1}^{\nu_2-1}\mE_{\mfx_t|\mathscr{N}=N}\!\Big[\! -\frac{1}{2}\|\mfu_t\|^2 \nonumber \\
&+ \! \frac{1}{2} \bar{\tilde{\mfx}}_t^T(H_t\backslash \mathfrak{R}_t) \bar{\tilde{\mfx}}_t \!+\! \frac{1}{2}\|\mathfrak{R}_t^{1/2}(\mfu_t \!+\! \mathfrak{R}_t^{-1}\mathfrak{S}_t\mfx_t \!+\! \mathfrak{R}_t^{-1}g_t)\|^2.\label{eq:ioc_objective_function_norm_quadratic}
\end{align}
Finally, note that since $\mathfrak{R}_t \succ 0$, the Schur complement $H_t\backslash \mathfrak{R}_t\succeq 0$ \citep[p.~495]{horn2013matrix}, and by also using \eqref{eq:stochastic_control_stand_still} it hence follows that
\begin{align}
\Psi(\cdot)&\ge \sum_{N = \nu_1}^{\nu_2}\Big\{\mP(\mathscr{N}=N)\sum_{t=1}^{\nu_2-1}\mE_{\mfx_t|\mathscr{N}=N}\Big[-\frac{1}{2}\|\mfu_t\|^2 \Big]\Big\} \nonumber\\
&=\sum_{t=1}^{\nu_2-1}\mE_{\mfx_t}\Big[-\frac{1}{2}\|\mfu_t\|^2\Big], \label{eq:proof_stoch_unique}
\end{align}
where the inequality follows by simply  adding the zeros $\sum_{t=1}^{\nu_2-N}\mE_{\mfx_t|\mathscr{N}=N}\Big[-\frac{1}{2}\|\mfu_t\|^2\Big]$ and removing all other nonnegative terms. This gives a lower bound for the optimal value.

Next, we show that the lower bound is actually attained by the ``true" $(\bar{Q},\bar{q}_{1:\nu_2},\bar{P}_{1:\nu_2},\bar{\eta}_{1:\nu_2})$.
First note that the Riccati iteration \eqref{eq:generalized_riccati_iterations} with the ``true" $\trueQ$ and $\trueq$ can be written as
\begin{subequations}
\begin{align}
&\begin{bmatrix}
\bar{P}_N &\bar{\eta}_N\\
\bar{\eta}_N^T &\bar{\xi}_N
\end{bmatrix}=
\begin{bmatrix}
\bar{Q} &\bar{q}\\
\bar{q}^T &0
\end{bmatrix},\label{eq:LMI_terminal}\\
&\mathbf{0} = \begin{bmatrix}
A^T\bar{P}_{t+1}A+\bar{Q}-\bar{P}_t &\bar{q}+A^T\bar{\eta}_{t+1}-\bar{\eta}_t\\
q^T+\bar{\eta}_{t+1}^TA-\bar{\eta}_t^T &\bar{\xi}_t
\end{bmatrix}  \label{eq:LMI_schur_complement} \\
&-\!\begin{bmatrix}
A^T\bar{P}_{t+1}B\\\bar{\eta}_{t+1}^TB
\end{bmatrix}
(B^T\bar{P}_{t+1}B+I)^{-1}
\begin{bmatrix}
B^T\bar{P}_{t+1}A &B^T\bar{\eta}_{t+1}
\end{bmatrix}\!,\nonumber
\end{align}
\end{subequations}
with a corresponding appropriate iteration for $\bar{\xi}_t$.  Moreover, as already observed earlier, it holds that $\bar{\mathfrak{R}}_t = B^T\bar{P}_{t+1}B+I\succ 0$ for all $t=1:N-1$ since $\bar{P}_t\in\mS_+^n$. 
Now, comparing \eqref{eq:Ht_Schur_comp} and \eqref{eq:LMI_schur_complement} we see that the latter is the Schur complement of $\bar{H}_t$ with respect to the top left corner, i.e., $\bar{H}_t\backslash \bar{\mathfrak{R}}_t$, and hence we have that $\bar{H}_t\backslash \bar{\mathfrak{R}}_t =0\succeq 0$. 
Therefore, by \citep[p.~495]{horn2013matrix}, this implies that $\bar{H}_t\succeq 0,\forall t=1:N-1$. Together with \eqref{eq:LMI_terminal} this implies that $\bar{Q}$, $\bar{q}$ and $\bar{P}_{1:N}$, $\bar{\eta}_{1:N}$ generated by Riccati iteration \eqref{eq:generalized_riccati_iterations} is a feasible solution to \eqref{eq:stochastic_IOC_opt_pro}.
Further, in view of \eqref{eq:ioc_objective_function_norm_quadratic} and the form of optimal control \eqref{eq:form_of-stochastic_u}, we can conclude that the lower bound $\sum_{t=1}^{\nu_2-1}\mE_{\mfx_t}\Big[-\frac{1}{2}\|\mfu_t\|^2\Big]$ is attained by the ``true" $(\bar{Q},\bar{q}_{1:\nu_2},\bar{P}_{1:\nu_2},\bar{\eta}_{1:\nu_2})$.

Finally, we show the uniqueness. To this end, suppose there exists another optimal solution $(Q^\prime=\trueQ+\Delta Q, \{q^\prime_t =\trueq_t - \Delta Qx_{t}^r \}_{t = 1}^{\nu_2},\{P_t^\prime=\bar{P}_t+\Delta P_t\}_{t=1}^{\nu_2},\{\eta_t^\prime=\bar{\eta}_t+\Delta\eta_t\}_{t = 1}^{\nu_2})$ that attains the same optimal value. In particular, this point must also attain the lower bound in \eqref{eq:proof_stoch_unique}. Moreover, by repeating the arguments for the inequality in \eqref{eq:proof_stoch_unique}, but removing all non-negative terms except for $\mathscr{N} = \nu_2$, we get
\begin{align*}
&\sum_{t=1}^{\nu_2-1}\mE_{\mfx_t}\Big[-\frac{1}{2}\|\mfu_t\|^2\Big]=\Psi(Q^\prime,q^\prime,P_{1:\nu_2}^\prime,\eta_{1:\nu_2}^\prime)\\
&\ge\sum_{t=1}^{\nu_2-1}\mE_{\mfx_t}\Big[-\frac{1}{2}\|\mfu_t\|^2\Big]+ \mP(\mathscr{N}=\nu_2)\sum_{t=1}^{\nu_2-1}\mE_{\mfx_t| \mathscr{N}=\nu_2}\Big[\\
&\frac{1}{2} \bar{\tilde{\mfx}}_t^T(H_t^{\prime}\backslash \mathfrak{R}_t^\prime)\bar{\tilde{\mfx}}_t \!+\! \frac{1}{2}\|(\mathfrak{R}_t^\prime)^{1/2}(\mfu_t \!+\! \mathfrak{R}_t^{\prime-1}\mathfrak{S}_t^\prime\mfx_t \!+\! \mathfrak{R}_t^{\prime-1}g_t^\prime)\|^2\Big].
\end{align*}
Since $H_t^{\prime}\backslash \mathfrak{R}_t^\prime\succeq 0,\forall t=1:\nu_2$, all terms in the last sum are nonnegative. Furthermore, since by assumption $\mP(\mathscr{N}=\nu_2)>0$, it must hold that all terms in the last sum are equal to zero. 
More precisely, we must have that
\begin{subequations}\label{eq:expectations_equal_zero}
\begin{align}
&\mE_{\mfx_t|\mathscr{N}=\nu_2}\left[\tilde{\bar{\mfx}}_t^T(H_t^{\prime}\backslash \mathfrak{R}_t^\prime)\bar{\tilde{\mfx}}_t\right]= \mE_{\mfx_t| \mathscr{N}=N}\left[ \tr((H_t^{\prime}\backslash \mathfrak{R}_t^\prime)\bar{\tilde{\mfx}}_t\bar{\tilde{\mfx}}_t^T)\right]\nonumber\\
&=  \tr\left((H_t^{\prime}\backslash \mathfrak{R}_t^\prime) \mE_{\mfx_t|\mathscr{N}=N}\left[\bar{\tilde{\mfx}}_t\bar{\tilde{\mfx}}_t^T\right] \right) = 0 \label{eq:stochastic_H_backslash_R} \\
&\mE_{\mfx_t|\mathscr{N}=\nu_2}\left[\|(\mathfrak{R}_t^\prime)^{1/2}(\mfu_t+\mathfrak{R}_t^{\prime-1}\mathfrak{S}_t^\prime\mfx_t+\mathfrak{R}_t^{\prime-1}g_t^\prime)\|^2\right]=0. \label{eq:stochastic_trajectory_match_2}
\end{align}
\end{subequations}
Since by Lemma \ref{lem:stochastic_persistent_excitation}, $\mE_{\mfx_t| \mathscr{N}=\nu_2}[\bar{\tilde{\mfx}}_t\bar{\tilde{\mfx}}_t^T]\succ 0$, in order for \eqref{eq:stochastic_H_backslash_R} to hold, we must have that $H_t^{\prime}\backslash \mathfrak{R}_t^\prime=0,\forall t=1:\nu_2$, which implies that the Riccati iterations \eqref{eq:generalized_riccati_iterations} must hold for all optimal solutions to \eqref{eq:stochastic_IOC_opt_pro} (cf.~\eqref{eq:LMI_schur_complement}).

To ease notation when analyzing expression \eqref{eq:stochastic_trajectory_match_2}, we temporarily introduce $\mathfrak{a}_t := \mfu_t+\mathfrak{R}_t^{\prime-1}\mathfrak{S}_t^\prime\mfx_t+\mathfrak{R}_t^{\prime-1}g_t^\prime$. From \eqref{eq:stochastic_trajectory_match_2} it then follows that
\begin{align*}
&0 = \mE_{\mfx_t|\mathscr{N}=\nu_2}\left[\|(\mathfrak{R}_t^\prime)^{1/2}\mathfrak{a}_t\|^2\right] = \mE_{\mfx_t|\mathscr{N}=\nu_2}\left[\mathfrak{a}_t^T \mathfrak{R}_t^\prime \mathfrak{a}_t^{} \right] \\
& =  \mE_{\mfx_t|\mathscr{N}=\nu_2}\left[\tr ( \mathfrak{a}_t^T \mathfrak{R}_t^\prime \mathfrak{a}_t^{} ) \right] = \tr ( \mathfrak{R}_t^\prime \mE_{\mfx_t|\mathscr{N}=\nu_2}\left[\mathfrak{a}_t^{} \mathfrak{a}_t^T\right]).
\end{align*}
Since $\mathfrak{R}_t^\prime \succ 0$ and $\mE_{\mfx_t|\mathscr{N}=\nu_2}\left[\mathfrak{a}_t^{} \mathfrak{a}_t^T\right] \succeq 0$, we must have that $\mE_{\mfx_t|\mathscr{N}=\nu_2}\left[\mathfrak{a}_t^{} \mathfrak{a}_t^T\right] = \mathbf{0}$. This means that
\begin{align*}
&0 = \tr(\mathfrak{R}_t^{\prime} \mathbf{0} \mathfrak{R}_t^{\prime T}) = \tr\left(\mathfrak{R}_t^{\prime} \mE_{\mfx_t|\mathscr{N}=\nu_2}\left[\mathfrak{a}_t^{} \mathfrak{a}_t^T\right] \mathfrak{R}_t^{\prime T} \right) \\
& = \mE_{\mfx_t|\mathscr{N}=\nu_2}\left[ \tr(\mathfrak{R}_t^{\prime}\mathfrak{a}_t^{} \mathfrak{a}_t^T \mathfrak{R}_t^{\prime T})\right] =
\mE_{\mfx_t|\mathscr{N}=\nu_2}\left[ \| \mathfrak{R}_t^{\prime}\mathfrak{a}_t \|^2 \right] \\
& =  \mE_{\mfx_t|\mathscr{N}=\nu_2}\left[ \| \mathfrak{R}_t^{\prime}\mfu_t+\mathfrak{S}_t^\prime\mfx_t+g_t^\prime \|^2 \right] \\
& = \mE_{\mfx_t|\mathscr{N}=\nu_2} [\| \mfu_t+B^TP_{t+1}^\prime(A\mfx_{t}+B\mfu_t)+B^T\eta_{t+1}^\prime \|^2],
\end{align*}
where the last equality comes from simply using the expressions for $\mathfrak{R}_t^{\prime}$, $\mathfrak{S}_t^\prime$, and $g_t^\prime$, as well as grouping the terms.
Moreover, by \eqref{eq:form_of-stochastic_u} we have that when conditioned on that $\mathscr{N}=\nu_2$, $\mfu_t = -B^T \bar{P}_{t+1}(A\mfx_{t}+B\mfu_t)-B^T\bar{\eta}_{t+1}$ for $t = 1:\nu_2-1$. Plugging this into the above gives that
\begin{align*}
0 & = \mE_{\mfx_t|\mathscr{N}=\nu_2} [\| B^T \Delta P_{t+1}(A\mfx_{t}+B\mfu_t)+B^T\Delta\eta_{t+1} \|^2] \\
& = \mE_{\mfx_t|\mathscr{N}=\nu_2} [\| B^T\begin{bmatrix}\Delta P_{t+1} &\Delta\eta_{t+1}\end{bmatrix} \bar{\tilde{\mfx}}_{t+1} \|^2],
\end{align*}
and rewriting this using the trace, as in \eqref{eq:stochastic_H_backslash_R}, as well as the fact that by Lemma \ref{lem:stochastic_persistent_excitation} we have that $\mE_{\mfx_t|\mathscr{N}=\nu_2}[\bar{\tilde{\mfx}}_{t+1}\bar{\tilde{\mfx}}_{t+1}^T]\succ 0$, we can conclude that
\[
 B^T\begin{bmatrix}\Delta P_{t+1} &\Delta\eta_{t+1}\end{bmatrix} = 0.
\]
Now, the uniqueness can be proved by following the same argument as in \citep[Thm.~2.1]{zhang2019inverse}.
\end{proof}

Note that in reality, the distributions of $\bar{\mfx}$, $\mfw_t$ and $\mathscr{N}$ are usually not a priori known, and hence it is not possible to calculate the objective function \eqref{eq:stochastic_IOC_objective_function} explicitly. Therefore, we cannot solve \eqref{eq:stochastic_IOC_opt_pro} directly, but instead need to empirically estimate the objective function.
To this end, let $M$ be the total number of trajectories observed, and let $M_N$ denote the number of observations which has $N$ time steps. Clearly, $\sum_{N = \nu_1}^{\nu_2} M_N = M$. First, we therefore approximate the expectation over $\mathscr{N}$ by using the empirical mean. Considering \eqref{eq:stochastic_ioc_obj_rewriting}, this is the same as estimating the probabilities $\mP(\mathscr{N}=N)$ using the empirical estimates  $M_N/M$.
Then, following along \citep[Sec.~4.2]{zhang2021inverse}, for each value of $N$ we approximate the inner expectation in \eqref{eq:stochastic_ioc_obj_rewriting} with the empirical mean
\begin{align*}
& \frac{1}{M_N} \sum_{i_N  = 1}^{M_N} \Big[\frac{1}{2}\mfx_{\nu_2}^{i_N T}P_{\nu_2}\mfx_{\nu_2}^{i_N}+\eta_{\nu_2}^T\mfx_{\nu_2}^{i_N} -\frac{1}{2}\mfx_{\nu_2-N+1}^{i_N T}P_{\nu_2-N+1}\\
&\times \mfx_{\nu_2-N+1}^{i_N} - \eta_{\nu_2-N+1}^T\mfx_{\nu_2-N+1}^{i_N} + \sum_{t=\nu_2-N+1}^{\nu_2-1} \Big( \frac{1}{2}\xi_t  \\
&+ \frac{1}{2} \mfx_t^{i_N T}Q\mfx_t^{i_N}  + q_t^T\mfx_t^{i_N} - \frac{1}{2}\tr(B^TP_{t+1}B\Sigma_w) \Big) \Big].
\end{align*}
Using both these approximations, we formulate the estimation problem as
\begin{align}
\min_{\substack{Q \in \mathbb{G}(\varphi)\\ \{ q_t \in \mR^n \}_{t=1:\nu_2} \\ \{P_{t} \in \mathbb{S}^n_+ \}_{t = 1:\nu_2},\\ \{ \eta_{t} \in\mathbb{R}^{n} \}_{t = 1:\nu_2},\\ \{ \xi_t \in \mR \}_{t=1:\nu_2-1}}} & \; \Psi^o(Q, q_{1:\nu_2},  P_{1:\nu_2}, \eta_{1:\nu_2},  \xi_{1:\nu_2-1})\nonumber\\
\st & \;  \text{\eqref{eq:stochastic_IOC_opt_pro_first_const}--\eqref{eq:stochastic_IOC_opt_pro_last_const} hold,}\label{eq:stochastic_ioc_approximation}
\end{align}
where
\begin{align*}
&\Psi^o(Q, q_{1:\nu_2},  P_{1:\nu_2}, \eta_{1:\nu_2},  \xi_{1:\nu_2-1}) = \frac{1}{M} \!\! \sum_{N = \nu_1}^{\nu_2} \sum_{i_N  = 1}^{M_N} \Big[ \\
&\frac{1}{2}\mfx_{\nu_2}^{i_NT}P_{\nu_2}\mfx_{\nu_2}^{i_N} + \eta_{\nu_2}^T\mfx_{\nu_2}^{i_N} -\frac{1}{2}\mfx_{\nu_2-N+1}^{i_N T}P_{\nu_2-N+1} \\
&\times \mfx_{\nu_2-N+1}^{i_N} -\eta_{\nu_2-N+1}^T\mfx_{\nu_2-N+1}^{i_N} +\sum_{t=\nu_2-N+1}^{\nu_2-1} \Big( \frac{1}{2}\xi_t\\
& + \frac{1}{2} \mfx_t^{i_N T}Q\mfx_t^{i_N}  + q_t^T\mfx_t^{i_N} - \frac{1}{2}\tr(B^TP_{t+1}B\Sigma_w) \Big) \Big].
\end{align*}

Next, we show that \eqref{eq:stochastic_ioc_approximation} is a statistically consistent estimator of the true parameter $\bar{Q}$. To this end, we first have the following two lemmas. For brevity, the details in the proofs are omitted.

\begin{lemma}\label{lem:bounded_domain_and_cost}
Let $\mathcal{D}$ be the set
\begin{align*}
& \{Q \in \mathbb{G}(\varphi),\; \{ q_t \in \mR^n \}_{t=1:\nu_2},\; \{P_{t} \in \mathbb{S}^n_+ \}_{t = 1:\nu_2}, \\
& \quad \{ \eta_{t} \in\mathbb{R}^{n} \}_{t = 1:\nu_2}, \;  \{ \xi_t \in \mR \}_{t=1:\nu_2-1} : \text{\eqref{eq:stochastic_IOC_opt_pro_first_const}--\eqref{eq:stochastic_IOC_opt_pro_last_const} hold}  \} .
\end{align*}
Then, $\mathcal{D}$ is compact and for any observed data the objective function $\Psi^o(Q, q_{1:\nu_2},  P_{1:\nu_2}, \eta_{1:\nu_2},  \xi_{1:\nu_2-1}) $ is bounded on $\mathcal{D}$.
\end{lemma}

\begin{proof}
This follows by adapting the arguments in \citep[Proof of Lem.~4.1]{zhang2021inverse}.
\end{proof}

\begin{lemma}[Uniform law of large numbers]\label{lem:ulln}
 Under Assumptions~\ref{ass:planning_horizon}, \ref{ass:AB}, \ref{ass:IID}, \ref{ass:persistent_excitation} and \ref{ass:bounded_parameter}, the optimal value
\begin{align*}
\sup & \; | \Psi^o(Q, q_t, P_t, \eta_t, \xi_t) -  \Psi(Q, q_t, P_t, \eta_t, \xi_t)| \\
\st & \; (Q, q_t, P_t, \eta_t, \xi_t) \in \mathcal{D}
\end{align*}
converges to $0$ almost surely as $M \to \infty$.
\end{lemma}

\begin{proof}
This follows by bounding the argument inside the expectation in $\Psi(Q, q_{1:\nu_2},  P_{1:\nu_2}, \eta_{1:\nu_2},  \xi_{1:\nu_2-1})$, given in \eqref{eq:stochastic_IOC_objective_function}, from above by an integrable function of the random variables and which is independent of the parameters $(Q, q_{1:\nu_2},  P_{1:\nu_2}, \eta_{1:\nu_2},  \xi_{1:\nu_2-1})$, and then applying \citep[Thm.~2]{jennrich1969asymptotic}. The former can be done via an argument similar to \citep[Proof of Lem.~4.2]{zhang2021inverse}, which relies on the bounds from Lemma~\ref{lem:stochastic_persistent_excitation}.
\end{proof}

\begin{theorem}[Statistical consistency]
Under Assumptions~\ref{ass:planning_horizon}, \ref{ass:AB}, \ref{ass:IID}, \ref{ass:persistent_excitation}, and \ref{ass:bounded_parameter}, the optimization problem \eqref{eq:stochastic_ioc_approximation} admits an optimal solution. For any such optimal solution $(Q_M^\star,q_{1:\nu_2}^\star, P_{1:\nu_2}^\star, \eta_{1:\nu_2}^\star,  \xi_{1:\nu_2-1}^\star)$, we have that $Q_M^\star \overset{p}\rightarrow \trueQ$ as $M\rightarrow \infty$.
\end{theorem}

\begin{proof}
It can be readily seen that \eqref{eq:stochastic_ioc_approximation} is convex, and hence any locally optimal solution is globally optimal. Moreover, by Lemma~\ref{lem:bounded_domain_and_cost}, the feasible domain of  \eqref{eq:stochastic_ioc_approximation} is compact and $\Psi^o$ is bounded on the feasible domain. Hence, \eqref{eq:stochastic_ioc_approximation} attains an optimal solution. The result now follows from \citep[Thm.~5.7]{van1998asymptotic}. In particular, since convergence almost surely implies convergence in probability \citep[Lem.~3.2]{kallenberg1997foundations}, Lemma~\ref{lem:ulln} implies that the first condition in the preceding theorem is satisfied. Moreover, the fact that $\mathcal{D}$ is compact (by Lemma~\ref{lem:bounded_domain_and_cost}) together with the fact that the optimal solution to \eqref{eq:stochastic_IOC_opt_pro} is unique (by Theorem~\ref{thm:stochastic_ioc_unique_opt_sol}) implies that the second condition of the preceding theorem holds. Hence the result follows.
\end{proof}

We end this section with a few remarks.

\begin{remark}
Most of the analysis still holds if the state terminal and running cost in the objective function of forward problem \eqref{eq:stochastic_forward_problem} is of the form $\frac{1}{2}\mfx_t^T\bar{Q}\mfx_t+\bar{q}^T\mfx_t$, i.e., with a constant unknown linear cost term $\trueq$ that does not relate to $\trueQ$ as in the tracking problem.
Extension to recover the unknown constant parameters $(\bar{Q},\bar{q})$ is straight-froward.
\end{remark}

\begin{remark}
The IOC algorithm \eqref{eq:stochastic_ioc_approximation} can easily be extended to the case where the observed data is obtained from experiments performed with a number of different reference signals for the tracking. The estimator is still formulated as the solution to a convex optimization problem akin to \eqref{eq:stochastic_ioc_approximation}, however for each reference signal we get a set of LMIs \eqref{eq:stochastic_IOC_opt_pro_first_const}--\eqref{eq:stochastic_IOC_opt_pro_last_const} as constraints.
\end{remark}

\section{Simulations and real-world experiments}
In this section we demonstrate the developed method, both on synthetic data in a simulation study and on real measurement data from an experimental set-up.
The simulation study uses the same underlying model as the experimental set-up, and hence we first describe the latter.
Since we do not know what the ``true" $\bar{Q}$ is for a real agent in practice, we create synthetic data for simulation to illustrate the statistical consistency of proposed algorithm. The performance of the algorithm in real practice is illustrated by first computing a predicted tracking trajectory based on the estimated $Q_{est}$ and then comparing the predicted trajectory with the trajectories collected in experiments.

\subsection{Experimental set-up}\label{sec:experiment_set_up}
We apply the IOC algorithm developed to identify a model for the locomotion of the human elbow.
Our interest in this set-up stem from its potential in the area of personalized rehabilitation robotics.
More precisely, for rehabilitation of hemiparesis,%
\footnote{Hemiparesis is a medical condition of reduced mobility or weakness in one side of the body that can occur, e.g., after stroke or trauma \citep{hemiparesis}.}
studies have shown that high-intensity, repeated training of specific tasks is an effective method to help patients recover some of the lost motor skills \citep{butefisch1995repetitive, kwakkel1997effects, kwakkel2002long, bayona2005role, hogan2006motions}.
In this setting, tracking a reference signal is one of the fundamental rehabilitation tasks,
but due to personal differences between the patients the tracking locomotion patterns could vary. This motivates personalized controller design for rehabilitation robots that help people in their rehabilitation therapy.
Under the assumption of optimal behavior (cf.~\citep{alexander1996optima}), the idea is to use IOC to identify the objective function used by the ``healthy'' side of the patient when performing certain tracking task. This objective function would reflect personal habits and preferences. The estimated objective function could then be used to design a control strategy for the rehabilitation robot that helps minimize the performance difference between the ``healthy'' and the ``unhealthy'' side.
\begin{figure}
\centering
\begin{subfigure}{.213\textwidth}
  \centering
  \includegraphics[width=1\linewidth]{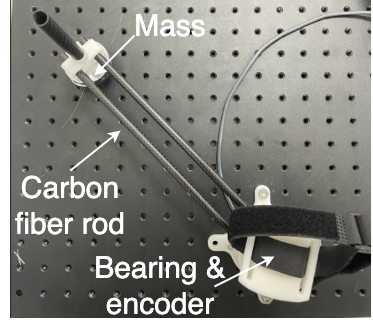}
  \caption{The experimental device}
  \label{fig:experiment_setup}
\end{subfigure}%
\begin{subfigure}{.287\textwidth}
  \centering
  \includegraphics[width=1\linewidth]{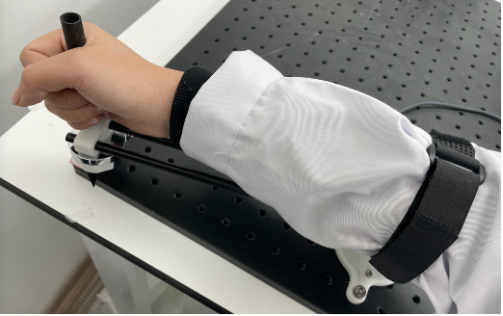}
  \caption{Human tracking locomotion}
  \label{fig:experiment_setup_1}
\end{subfigure}
\caption{
(\ref{fig:experiment_setup}): A mass of 0.2kg is attached to a light carbon fiber rod and can rotate around the axis. The rotation is performed via a bearing to eliminate potential friction as much as possible. The distance between the centre of the mass and the axis is 0.255 metres. The angular position and velocity of the mass are measured every 0.05 seconds by an encoder attached to the axis.
(\ref{fig:experiment_setup_1}): During the experiment, the test subject is using the device from Fig.~\ref{fig:experiment_setup} as shown in this image. Moreover, two sliding blocks that correspond to the reference signal and the state of the experimental device are shown on the computer screen. The human holds the handle and rotate his/her elbow to control the corresponding sliding block on the screen so that it tracks the reference signal.}
\label{fig:test}
\end{figure}

The experimental device used is illustrated in Fig.~\ref{fig:test}. With this device set-up, the ideal continuous-time dynamics of the mass is given by
\begin{align*}
\dot{x}=\underbrace{
\begin{bmatrix}
0 &1\\0 &0
\end{bmatrix}}_{\hat{A}}x+\underbrace{
\begin{bmatrix}
0\\\frac{1}{I}
\end{bmatrix}}_{\hat{B}}u,
\end{align*}
where $x=[x_1,x_2]^T$, and where $x_1$, $x_2$, $u\in\mathbb{R}$ denote the angular position, the angular velocity, and the torque applied to the axis, respectively. Moreover, $I=m\ell^2$ denotes the moment of inertia, where $m=0.2$ and $\ell=0.255$ denote the mass (in kilograms) and the distance (in meters) between the center of the mass and the axis, respectively. Since measurements are taken every 0.05 seconds, we discretize the ideal system dynamics in time and obtain the discrete time system $(A,B)$, where $A=e^{\hat{A}\Delta t}$, $B=\int_0^{\Delta t}e^{\hat{A} t}dt\hat{B}$, and where $\Delta t=0.05$. However, due to the existence of friction in the bearing (albeit small), as well as a certain amount of intrinsic randomness in human locomotion, in the discrete-time model we add a process noise $\mfw_t$ to the control input and hence get a discrete-time system in the form of \eqref{eq:stochastic_forward_problem_dynamics}. 

Notably, the aforementioned discrete-time dynamics is mathematically the same as the model for pushing a point mass along a one-dimensional line. Hence, in the experimental set-up the state $\mfx_t$ is illustrated on a computer screen by a corresponding block sliding back and forth in one dimension. The movement of the block can be controlled by operating the experiment device, as illustrated in Fig.~\ref{fig:test}.
In the elbow rehabilitation tracking task, we set $\nu_2 = 120$ and generate an oscillating reference signal $x_{1:120}^r$ as
\begin{align}
x_{t+1}^r &= Ax_t^r+Bu_t^r, \quad t=1:\nu_2=120, \label{eq:ref_signal}\\
x_1^r &=[0, -0.5]^T, \quad u_t^r = 0.01\sin(\frac{\pi}{40}\cdot t) \nonumber.
\end{align}
This defines the behavior of the ``target"  block to track on the computer screen. The reference signal is shown multiple times to the human before the test starts, so that he/she is totally aware of the reference signal. 
Once the test starts, the reference signal is played on the computer screen. The human can choose when to start the tracking, and the data is recorded until the final time instant $\nu_2=120$.

Furthermore, to measure the mean and covariance of the process noise, i.e., to estimate $\mE[\mfw_t]$ and $\Sigma_w$, we construct another reference signal. This signal moves at a constant speed, and we let the human track it. The rationale of such set-up lies in the fact that, once the task of tracking a constant velocity reference signal gets close to a ``steady state'', the control input $\mfu_t$ should be close to zero. Hence any fluctuations of the velocity realizations $x_{t,2}$ can be modeled as caused merely by the process noise $\mfw_t$. Thus, the realizations of the process noise $\mfw_t$ are calculated by $w_t^i = B^\dagger(x_{t+1}^i-Ax_t^i)$, where $^\dagger$ denotes the Moore-Penrose generalized inverse. 
We use the standard empirical estimation to estimate the mean $\mE[\mfw_t]$ as well as the covariance $\Sigma_w$, and obtain $\mE[\mfw_t]\approx 3.3667\times 10^{-4} \approx 0$ and $\Sigma_w\approx 6.8062\times 10^{-4}$.

Finally, to solve the estimation problem \eqref{eq:stochastic_ioc_approximation} we implement the problem in Matlab using YALMIP \citep{lofberg2004yalmip} and then solve it numerically using MOSEK \citep{mosek}.

\subsection{Simulation results}\label{sec:sim_results}
To show the performance of our IOC algorithm and also illustrate the statistical consistency, we use the same parameters as in the experiment set-up described in Sec.~\ref{sec:experiment_set_up} and simulate $M=5000$ observed trajectories. 
The length of each trajectory, denoted $N^i$, is generated by sampling from a uniform distribution on the integers in $[80,120]$.
For each simulated trajectory, the angular position from which the tracking starts is generated by $x_{120-N^i+1,1}=x^r_{120-N^i+1,1}+\varepsilon^i$, where $\varepsilon^i$ is sampled from a uniform distribution on $[-\frac{\pi}{6},\frac{\pi}{6}]$, and the angular velocity from which the tracking starts is set to zero, i.e., $x_{120-N^i+1,2}=0$. The ``true" parameter $\bar{Q}$ in the objective function is set to $0.01I$. In the reconstruction, we set $\varphi$ in Assumption~\ref{ass:bounded_parameter} to $50$.
Relative errors of the estimates obtained using the developed algorithm are shown in Fig.~\ref{fig:statistical_consistency}, and as can be seen the relative estimation error of the parameter decreases as $M$ increases, illustrating the statistical consistency.

\begin{figure}[!htpb]
\centering
\includegraphics[width=0.4\textwidth, trim=6 6 30 4, clip]{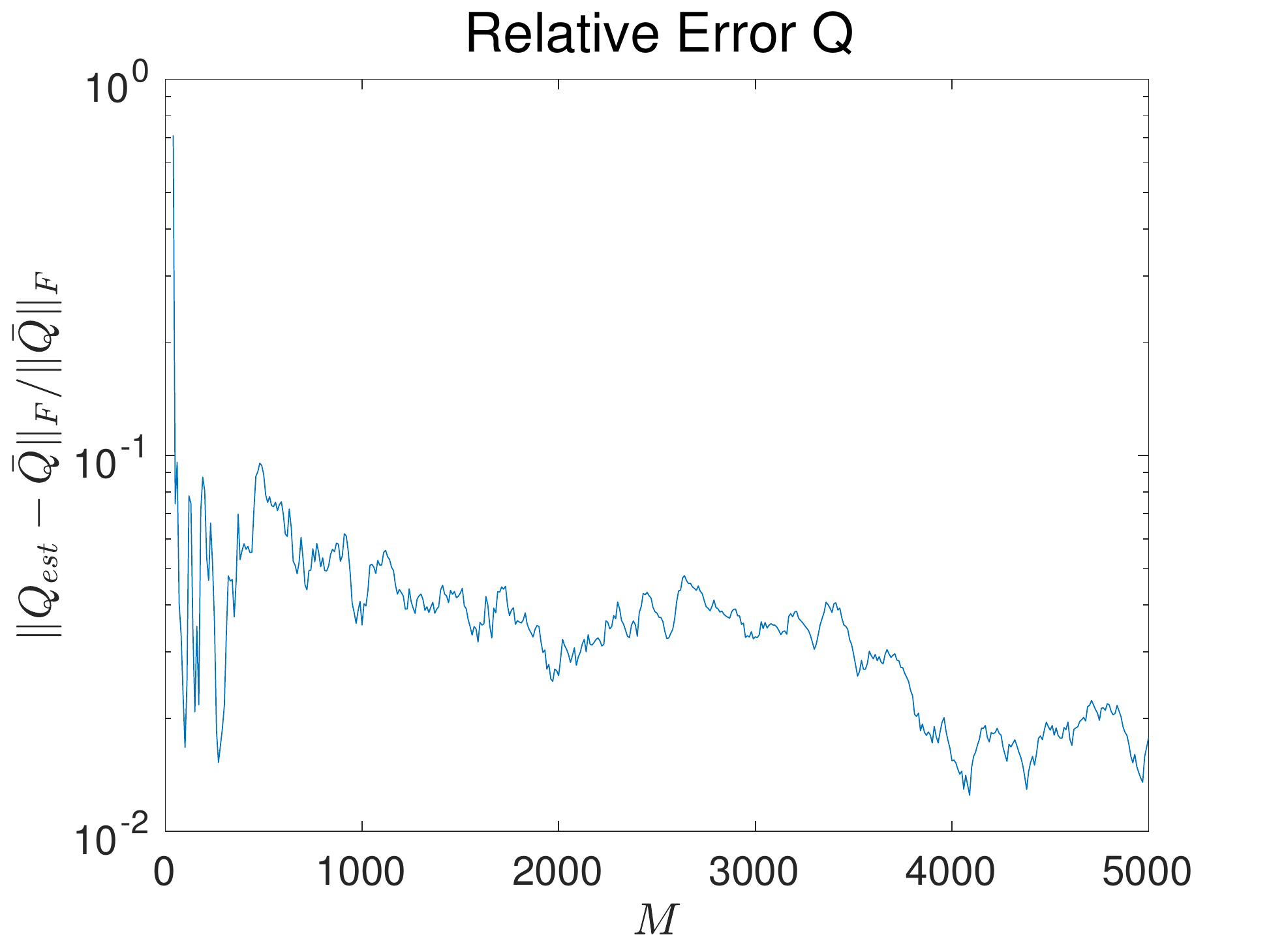}
\caption{The relative estimation error vs.~the number of observed trajectories $M$ for the simulation experiment described in Sec.~\ref{sec:sim_results}.}
\label{fig:statistical_consistency}
\end{figure}

\subsection{Experimental results}\label{sec:exp_results}
We apply the developed algorithm to data from experiments using the experiment set-up described in Sec.~\ref{sec:experiment_set_up}. More precisely, we collect $M=500$ trajectories from such tracking experiments. The parameter estimate $Q_{est}$ obtained, based on the human tracking data, is
\[
Q_{est}=
\begin{bmatrix}
0.0043    &0.0017\\
0.0017  &0.0007
\end{bmatrix}.
\]

To verify the validity of this estimate, we compare the predicted tracking behavior with measured tracking behavior on two new reference signals. In particular, the new reference signals are generated as in \eqref{eq:ref_signal}, but where the first one is generated by $u_t^{r_1} = 0.01\cos(\tfrac{\pi}{40}t)$ and $x_1^{r_1} = [0,0.05]^T$, and second is generated by $u_t^{r_2} = 0$ and $x_1^{r_2} = [0,0.2]^T$. Notably, the new reference signals are \emph{different} from the one that is used to collect the training data used in the IOC algorithm.
With these new reference signals, tracking experiments are performed and data is collected. 
Nevertheless, in order to evaluate the prediction accuracy despite the process noise, in each tracking experiment performed we let the human start the tracking at the same angular position each time (and with a zero angular velocity).
To obtain the predicted tracking behavior, we numerically solve the forward optimal control problem using the estimated parameter $Q_{est}$ and the two new reference signals, but without the process noise. The process noise is removed from the generation of the prediction in order to facilitate the comparison, since the form of the control signal in \eqref{eq:form_of-stochastic_u} is not explicitly dependent on the process noise. For each reference signal and planing horizon we hence obtain a predicted trajectory $x_{\nu_2-N+1:\nu_2}^{pred}$. The predicted trajectory and the actual collected trajectories, for both reference signals, are illustrated in Fig.~\ref{fig:verification_result}. As can be seen, the prediction based on estimated parameter $Q_{est}$ is quite accurate in the sense that the actual human tracking locomotion is well-predicted by the trajectories computed by the corresponding $Q_{est}$.
This indicates that the model proposed in the paper describes actual human tracking locomotion well, and thus provide a good model for rehabilitation robots controller design in personalized rehabilitation.
\begin{figure}[!htpb]
\centering
\includegraphics[width = 0.5\textwidth, trim=25 4 30 2, clip]{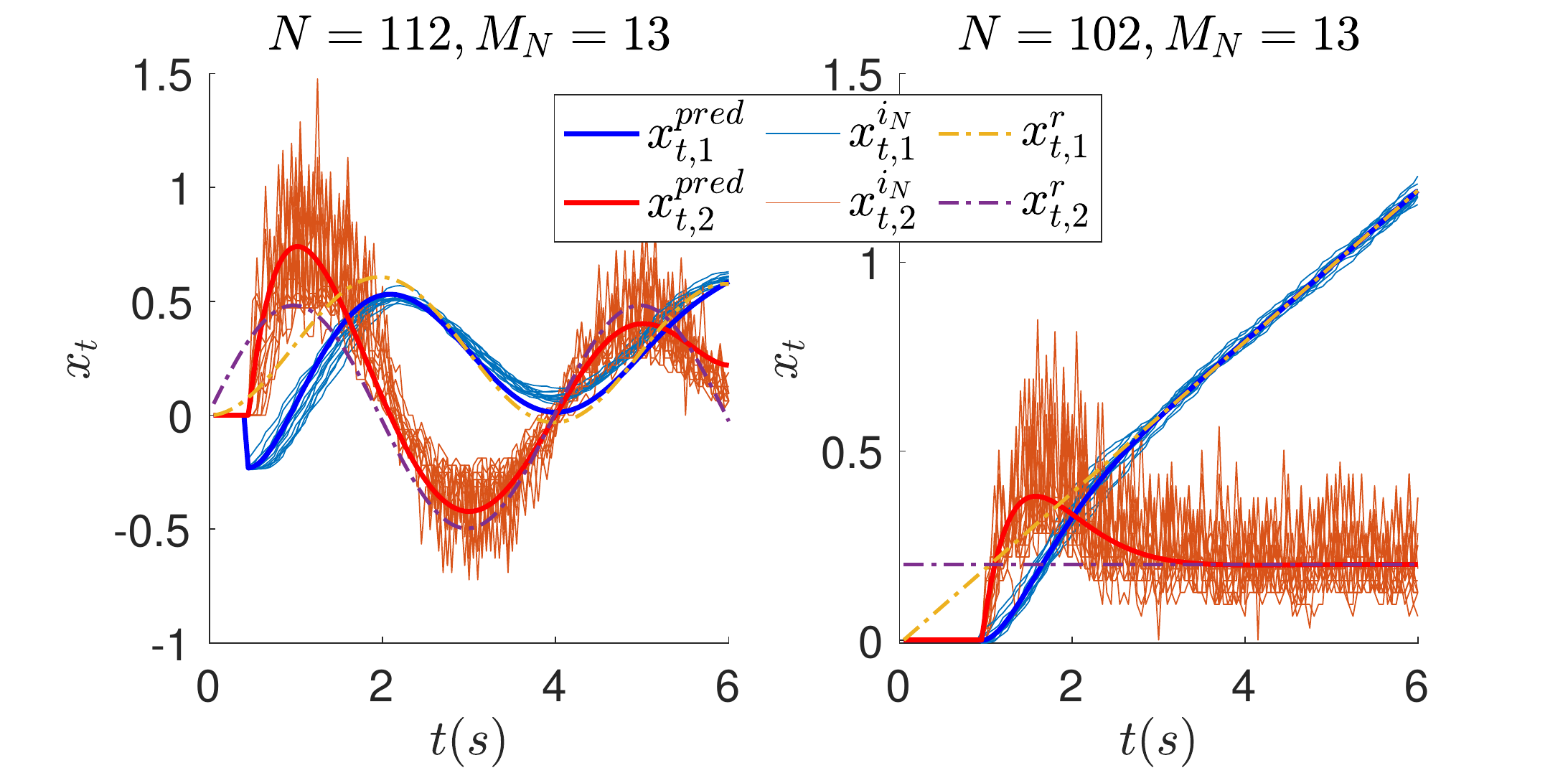}
\caption{Illustration of the predicted signals and the measure signals, as described in Sec.~\ref{sec:exp_results}. Due to space limitation, we only show the cases of the planing horizon $N=112$ (for which $M_N=13$ ) for $x_t^{r_1}$ and the planing horizon $N=102$ (for which $M_N=13$) for $x_t^{r_2}$. Similar results are obtained for the other planning horizon lengths.}
\label{fig:verification_result}
\end{figure}

\section{Conclusion and future work}
IOC provides a powerful theoretical framework for modelling and identifying expert systems. Namely, it enables data-driven objective function design in optimal control, where ``expert data" can be used to develop a control strategy that suits the contextual environment. As a conclusion:
\begin{enumerate}
\item We develop a statistically consistent IOC algorithm for linear-quadratic tracking problems with random planing horizons, which has not been considered before. Treating the planning horizon as a random variable allows for a systematic way to handle observed data records of different time lengths, while still retaining guarantees about statistical consistency. The statistical consistency ensure the robustness of the estimate and has not been considered in existing IOC frameworks that are more general, such as \citep{keshavarz2011imputing, hatz2012estimating, pauwels2016linear, rouot2017inverse, molloy2018finite, molloy2020online}.
\item The proposed IOC algorithm is based on convex optimization. The convexity guarantees that the estimate obtained is the global optima, and hence that the theoretical statistical consistency is actually achieved in real practice. This has not been achieved in existing work, such as \citep{molloy2018finite, molloy2020online, zhang2019inverse}.
\item The proposed IOC algorithm is implemented and verified on both simulated and real data. The former data is used to demonstrate the statistical consistency in practice, while the latter is used to demonstrate the capabilities of the method in applications. More precisely, the data is collected in an experimental set-up with a human tracking task motivated from rehabilitation robotics, and the results show that the actual human tracking locomotion is well-predicted by the corresponding estimated objective function -- even when the reference signals are different from the one that is used for training.
\end{enumerate}

Future work includes extending the results to more sophisticated nonlinear IOC problems and applying the results to the controller design for actual rehabilitation robots.

\bibliographystyle{ieeetr}
\bibliography{ref}
\end{document}